\documentclass[12pt,a4paper]{amsart}

\usepackage{amsmath}
\usepackage{amssymb}
\usepackage{amsthm}
\usepackage{color}
\usepackage{enumerate}
\usepackage[dvips]{graphicx}
\usepackage{float}
\usepackage{wrapfig}
\usepackage{layout}
\usepackage{comment}
\newcommand{\RN}{\mathbb{R}^N}

\def\R{\mathbb{R}}

\def\cF{\mathcal{F}}

\def\bye{\end{document}}
\def\by{\end{proof}\bye}

\def\hello{\begin{document}}
\def\fr{\frac} 
\def\disp{\displaystyle}  
\def\ga{\alpha}     
\def\go{\omega}
\def\gep{\varepsilon}      
\def\ep{\gep}    
\def\mid{\,:\,}   
\def\gb{\beta} 
\def\gam{\gamma}
\def\gd{\delta}
\def\gz{\zeta} 
\def\gth{\theta}   
\def\gk{\kappa} 
\def\gl{\lambda}
\def\gL{\Lambda}
\def\gs{\sigma}   
\def\gf{\varphi}                  
\def\tim{\times}                        
\def\aln{&\,}
\def\ol{\overline}
\def\ul{\underline}           
\def\pl{\partial}
\def\hb{\text}                
\def\cF{\mathcal{F}}
\def\Int{\mathop{\text{int}}}

\def\gG{\varGamma}
\def\lan{\langle}
\def\ran{\rangle}
\def\cD{\mathcal{D}}
\def\cB{\mathcal{B}}
\def\bcases{\begin{cases}}
\def\ecases{\end{cases}}
\def\balns{\begin{align*}}
\def\ealns{\end{align*}}
\def\balnd{\begin{aligned}}
\def\ealnd{\end{aligned}}
\def\1{\mathbf{1}}
\def\bproof{\begin{proof}}

\def\eproof{\end{proof}}

\theoremstyle{definition}
\newtheorem{definition}{Definition}
\theoremstyle{plain}
\newtheorem{theorem}[definition]{Theorem}
\newtheorem{corollary}[definition]{Corollary}
\newtheorem{lemma}[definition]{Lemma}
\newtheorem{proposition}[definition]{Proposition}
\theoremstyle{remark}
\newtheorem{remark}[definition]{Remark}
\newtheorem{notation}[definition]{Notation}

\def\red#1{\textcolor{red}{#1}}
\def\blu#1{\textcolor{blue}{#1}}
\def\beq{\begin{equation}}
\def\eeq{\end{equation}}
\def\bthm{\begin{theorem}}
\def\ethm{\end{theorem}}
\def\bproof{\begin{proof}}
\def\eproof{\end{proof}}

\newcommand{\lbar}[1]{\mkern 1.9mu\overline{\mkern-1.9mu#1\mkern-0.1mu}
\mkern 0.1mu}
\long\def\/*#1*/{}
\def\Gth{\varTheta}
\def\gD{\varDelta}
\def\gG{\varGamma}
\def\what{\widehat}
\def\hM{\what M}
\def\eqr#1{\eqref{#1}}
\def\tin{\hbox{ in }}
\def\ton{\hbox{ on }}
\def\bmat{\begin{pmatrix}}
\def\emat{\end{pmatrix}}

\newcommand{\Pmo}{\mathcal{P}^-_{1}}
\newcommand{\Ppo}{\mathcal{P}^+_{1}}
\newcommand{\Pmk}{\mathcal{P}^-_{k}}
\newcommand{\Ppk}{\mathcal{P}^+_{k}}

\newcommand{\Rn}{{\mathbb R}^N}

\title{Towards a reversed Faber-Krahn inequality for the truncated Laplacian}
\author{Isabeau Birindelli}
 \address[I. Birindelli]{Dipartimento di Matematica \lq\lq G. Castelnuovo\rq\rq\newline
\indent Sapienza Universit\`a  di Roma, P.le Aldo  Moro 2, I--00185 Roma, Italy.\newline
\indent isabeau@mat.uniroma1.it}
\author{Giulio Galise}
 \address[G. Galise]{Dipartimento di Matematica \lq\lq F. Enriques\rq\rq\newline
\indent  Universit\`a  di Milano \\
 \indent   Via Cesare Saldini 50 - 20133 Milan, Italy.\newline
\indent Giulio.Galise@unimi.it}
\author{Hitoshi Ishii}
\address[H. Ishii]{Faculty of education and Integrated Arts and Sciences\newline
\indent Waseda University\\
 \indent 1-6-1 Nishi-Waseda, Shinjuku, Tokyo 169-8050 Japan.\newline
\indent hitoshi.ishii@waseda.jp}

\begin{document}
\begin{abstract}
We consider the nonlinear eigenvalue problem, with Dirichlet boundary condition, for a class of very degenerate elliptic operators, with the aim to show that, at least for square type domains having fixed volume, the symmetry of the domain maximize the principal eigenvalue, contrary to what happens for the Laplacian.

\end{abstract}

\maketitle

\section{Introduction}\label{intro}
Let us recall that if $\Omega$ is a strictly convex domain and $\lambda_N(X)$ indicates the largest eigenvalue of the symmetric matrix $X$ then there exists $\mu_1^+>0$ and $\varphi(\cdot)>0$ in $\Omega$ such that
$$\left\{\begin{array}{rl}
\lambda_N(D^2\varphi)+\mu_1^+\varphi=0 & \mbox{in}\ \Omega\\
\varphi =0 & \mbox{on}\ \partial\Omega.
\end{array}
\right.
$$
This was proved in \cite{BGI}. With a little abuse, but for obvious reasons, we called $\mu_1^+$ and $\varphi$ 
respectively the principal eigenvalue and eigenfunction for the 
operator $\Ppo(D^2u)=\lambda_N(D^2u)$ in $\Omega$. The value $\mu_1^+$ shares many features with 
$\mu(\Delta)$ 
the principal  eigenvalue of the Laplacian with homogenous Dirichlet conditions, e.g. the fact that $\mu_1^+$ is a 
barrier for the validity of the maximum principle. But,  strikingly, also many differences.  
We naturally wondered if other qualitative properties could be extended from $\mu(\Delta)$ to $\mu^+_1$.

Let us start by stating our most surprising result. 
\begin{quote}
\lq\lq \emph{Among rectangles with given measure the square has the {\bf largest} eigenvalue $\mu_1^+$ and the eigenvalue of the ball of same measure will be even larger than that of the square.}\rq\rq
\end{quote}
This is surprising since, as it is well known, on the contrary, for $\mu(\Delta)$ the principal eigenvalue of the 
Laplacian, the Faber-Krahn inequality states that
\begin{quote} 
\lq\lq \emph{Among domains with given measure the ball has the {\bf smallest} eigenvalue $\mu(\Delta)$}\rq\rq
\end{quote}
which, in its much weaker form, reduces to the obvious fact
\begin{quote}
\lq\lq \emph{Among rectangles with given measure the square has the {\bf smallest} eigenvalue $\mu(\Delta)$.}\rq\rq
\end{quote}
In \cite{BGI} we consider a more general class of operators, sometimes called \emph{truncated Laplacian}, which we now describe.
For any $N\times N$ symmetric matrix $X$, let 
\begin{equation}\label{arranged eigenvalues}
\lambda_1(X)\leq \lambda_2(X)\leq\cdots\leq \lambda_N(X) 
\end{equation}
be the ordered eigenvalues of $X$. For $k\in[1,N]$, $k$ integer, let 
\begin{equation}\label{def1}
\Pmk(D^2u)=\sum_{i=1}^k\lambda_i(D^2u) \qquad\text{and}\qquad \Ppk(D^2u)=\sum_{i=1}^k\lambda_{N+1-i}(D^2u).
\end{equation}
For $k=N$ these operators coincide with the Laplacian, hence we will always consider $k<N$. 
We want to emphasize that they are fully nonlinear elliptic operators that are degenerate at every point and in every direction. 

The truncated Laplacian initially appears in Sha \cite{Sha1,Sha2} and Wu \cite{Wu} in order to investigate compact 
manifolds having \emph{$k$-convex} boundary, i.e. such that the sum of any $k$ principal curvature functions is 
positive. Later the operators ${\mathcal P}^\pm_k$ can be found in \cite{AS}, where Ambrosio and Soner developed a 
level set theory to the  the mean curvature evolution of surfaces with arbitrary codimension. More recently we wish to 
recall the theory of subequations of Harvey and Lawson, see e.g.  \cite{HL1, HL2}, which give a new geometric 
interpretation of  solutions,  and the works of Caffarelli, Li and Nirenberg \cite{CLN1, CLN3}  concerning removable 
singularities along smooth manifolds for Dirichlet problems associated to $\Pmk$. The 
extended version of the maximum principle and the study of positive solutions has been done in \cite{AGV,GV,G}, see also \cite{CDLV} in the case of entire solutions. The case $k=1$ is treated in the nice paper of Oberman and Silvestre \cite{OS} about convex envelope. 
Blanc and Rossi in \cite{BR} consider a similar class of operators, when one takes just one eigenvalue of the Hessian 
matrix, but not necessarily the first or last one.

Following Berestycki, Nirenberg, Varadhan \cite{BNV}, one can define a \lq\lq candidate\rq\rq\ for the principal 
eigenvalue:
$$\mu_k^-=\sup\{\mu\in\R, \ \exists\ \phi>0\ \mbox{in}\ \Omega, \Pmk(D^2\phi)+\mu \phi\leq 0\},$$
or
$$\mu_k^+=\sup\{\mu\in\R, \ \exists\ \phi>0\ \mbox{in}\ \Omega, \Ppk(D^2\phi)+\mu \phi\leq 0\}.$$
Interestingly, $\mu_k^-=+\infty$ for any bounded domain $\Omega$, while $\mu_k^+<+\infty$. 
Hence we will concentrate on the latter. As recalled above, in \cite{BGI} the existence of an eigenfunction was done 
only for $k=1$ and when $\Omega$ is strictly convex. 

\medskip
Observe that studying $\mu_1^+$ in rectangles had a triple interest, on one hand we wished to see, in the simplest 
case, if the strict convexity was a necessary condition for the existence of the eigenfunction. 
On the other hand we hoped  to construct eigenfunctions for $k>1$. Finally, it was a way to see if one could expect 
some relationship between the symmetry of the domain and the size of the principal eigenvalue, as in Faber-Krahn 
inequalities. We shall now discuss what we have obtained in these three directions.

\medskip
On this third point we have seen at the beginning that one should, if anything, expect a reversed Faber-Krahn 
inequality. We wish to point out another feature that cannot be extend from $\mu(\Delta)$ to $\mu_1^+$, it is the famous result of Lieb. He showed,  in \cite{L}, that if $A, B\subset\Rn$ are two bounded domains, then 
\begin{equation}\label{L0}
\inf_{x\in\Rn}\mu\left(\Delta, A\cap B_x\right)<\mu(\Delta, A)+\mu(\Delta, B),
\end{equation}
 $\mu(\Delta,\Omega) $ being the principal eigenvalue of the Laplacian with Dirichlet boundary conditions in $\Omega$ and $B_x=x+B$ denoting $B$ translated by $x\in\Rn$.

The inequality \eqref{L0} is not true in general for $\mu^+_1$ , actually it is reversed if $A$ and $B$ are some specific 
rectangles. 

\medskip
Concerning the first point, the remark we need to make is that, even though rectangles are not strictly convex, in Theorem \ref{Faber-Krahn}  
we construct explicitly an eigenfunction and its corresponding eigenvalue;
the eigenfunction is a product of functions of one variable. The proof is not at all obvious but it uses only elementary   
tools from linear algebra and ode.

The question of whether the condition on the strict convexity is necessary for the 
existence of the eigenfunction $\varphi$ was raised in \cite{BGI}. It was related in particular with the fact that 
we could prove global Lipschitz  regularity for the Dirichlet problem under that hypothesis. 

Let us observe that the eigenfunctions that we construct are indeed only H\"older 
continuous up to the boundary which confirms that in general, in order to get Lipschitz regularity up to the boundary, 
the hypothesis of the strict convexity cannot be removed.
In this paper, thanks to the eigenfunctions in squares that we have constructed, we extend the regularity results to domains that are convex but not necessarily strictly convex. Indeed in that case, we shall prove that,
under the condition that near the boundary the forcing term 
is not \lq\lq too\rq\rq\ negative, the solution of the Dirichlet problems exists and it  is  H\"older continuous up to the 
boundary.  This is done in Theorem \ref{holder}.  

On the other hand it is not clear if the condition which we require on the forcing 
term is necessary. 
For example, suppose that  $f\leq-1$ in  some domain ${\Omega}$  which is not strictly convex; can we expect  that 
there are  solutions of
$$\Ppo(D^2u)=f \;\ \mbox{in}\ \Omega,\ u=0  \ \mbox{on}\ \partial\Omega?$$
We expect the answer to be negative.
\medskip

Concerning $k\geq1$, remark that if $\Omega=B_\rho\subset\RN$ we can construct the eigenvalue $\mu^+_k$ of 
$\Ppk$ and a corresponding eigenfunction in term of those of the Laplacian in space dimension $k$. 

Let $\phi(x):=v(|x|)$ and $\mu(\Delta)$ be respectively the eigenvalue and the eigenfunction of the Laplacian in the 
ball of radius $\rho$ in ${\mathbb R}^k$. Hence $v$ satisfies:
\begin{equation}
\left\{
\begin{array}{rl}
v''(r)+\frac{k-1}{r}v'(r)+\mu(\Delta) v(r)=0 & \text{for $r\in(0,\rho)$}\\
v'(0)=0,\;v(\rho)=0.\\
\end{array}
\right.
\end{equation}
Since $v'\leq0$, arguing as in \cite{BGL}: 
$$
\left(v''(r)-\frac{v'(r)}{r}\right)'\geq-\frac{k}{r}\left(v''(r)-\frac{v'(r)}{r}\right)
$$
and $v''(r)\geq\frac{v'(r)}{r}$ for any $r\in(0,\rho)$. Set $u(x)=v(|x|)$ for $x\in B_\rho$, then
$$
\Ppk(D^2u(x))+\mu(\Delta) u(x)=v''(|x|)+\frac{k-1}{|x|}v'(|x|)+\mu(\Delta) v(|x|)=0,\qquad x\in B_\rho.
$$
This implies that
\begin{equation}\label{Pk-Delta}
\mu^+_k=\mu(\Delta)
\end{equation} and answers the question that there are at least some domains for which the
principal eigenfunctions exists even for $k>1$.
On the other hand, for the rectangles we don't know if there is a corresponding eigenfunction. Indeed, contrarily to the 
case $k=1$ and $k=N$, we prove that for $k=2$, $\ldots$, $N-1$,
if it exists, the eigenfunction cannot be a function which is the product of functions of one variable. 
It is worth pointing out that other fully nonlinear operators for which this is true, 
are the Pucci extremal operators. This was proved in \cite{BL}.

\medskip

The paper is organized in the following way. The next section is preliminary, instead in section \ref{efs} we 
construct the explicit eigenfunctions for $k=1$ and we treat also the case $k>1$. Section \ref{regh} is devoted to existence and the H\"older regularity in convex 
domain of the Dirichlet problem.

\section{Preliminaries}\label{prel}
We denote by $\mathbb S^N$ the set of   $N\times N$ symmetric real matrices equipped with its usual partial order. The eigenvalues of $X\in\mathbb S^N$ will be henceforth arranged in the nondecreasing order \eqref{arranged eigenvalues}. The norm of $X$ is
$$
\left\|X\right\|=\max_{i=1,\ldots,N}\left|\lambda_i(X)\right|.
$$
The operators ${\mathcal{P}}^\pm_k$, which are fully nonlinear degenerate elliptic operators, can be equivalently defined either by the partial sums \eqref{def1} or by the representation formulas
\begin{equation}\label{representation formulas}
\begin{split}
\Pmk(X)&=\min\left\{\sum_{i=1}^k\left\langle Xv_i,v_i\right\rangle\,|\,\text{$v_i\in\Rn$ and $\left\langle v_i,v_j\right\rangle=\delta_{ij}$, for $i,j=1,\ldots,k$}\right\}\\
\Ppk(X)&=\max\left\{\sum_{i=1}^k\left\langle Xv_i,v_i\right\rangle\,|\,\text{$v_i\in\Rn$ and $\left\langle v_i,v_j\right\rangle=\delta_{ij}$, for $i,j=1,\ldots,k$}\right\}.
\end{split}
\end{equation}
From \eqref{representation formulas} one deduce the inequalities
$$
\Pmk(X-Y)\leq{\mathcal P}^\pm_k(X)-{\mathcal P}^\pm_k(Y)\leq\Ppk(X-Y)
$$
and the Lipschitz continuity of ${\mathcal P}^\pm_k:\mathbb S^N\mapsto\R$: for any $X,Y\in\mathbb S^N$ 
\begin{equation}\label{continuity}
\left|{\mathcal P}^\pm_k(X)-{\mathcal P}^\pm_k(Y)\right|\leq k\left\|X-Y\right\|.
\end{equation}

The following elementary linear algebra Lemma will play a key role. 
\begin{lemma}\label{linear algebra}
Let $a,b\in\R$ and let us consider the symmetric matrix
\begin{equation}\label{matrix M}
M(a,b)=
\left(
\begin{array}{ccccc}
a & b & b & \ldots & b\\
b & a &   b & \ldots & b\\
\vdots & \vdots & \ddots & \vdots & \vdots\\
b & \ldots & b & a &b\\
b & b & \ldots &  b & a
\end{array}
\right).
\end{equation}
Then, for $b\neq 0$, the eigenvalues of $M(a,b)$ are
\begin{itemize}
	\item $a-b$ with multiplicity $N-1$ and and its eigenspace is $V=\left\{x\in\Rn\,:\,\sum_{i=1}^Nx_i=0\right\}$;
	\item $a+(N-1)b$ which is simple and its eigenspace $V^\bot$ spanned by $(1,\ldots,1)^\mathsf{T}$.
\end{itemize} \end{lemma}

\section{Construction of eigenfunctions.}\label{efs}

\subsection{Cube} Let $Q_{2R}$ be the $N$-dimensional open cube with center $0$ and side length $2R$, i.e.
$$
Q_{2R}=(-R,R)^N\,.
$$ 
We start by computing the principal eigenvalue $\mu^+_1$ in $Q_{2R}$ for the operator $\Ppo$ by constructing a 
positive eigenfunction having the multiplicative form 
\begin{equation}\label{multiplicative form}
u(x)=\prod_{i=1}^N f(x_i)\qquad x\in Q_{2R},
\end{equation}
with $f$ a positive smooth function to be determined. By homogeneity we assume $u(0)=1$, hence $f(0)=1$. To find out $f$, we compute
$$
\begin{array}{ll}
\displaystyle\partial_{ii}u(x)=f''(x_i)\prod_{k\neq i}f(x_k)&\text{for $i=1,\ldots,N$}\\
\displaystyle\partial_{ij}u(x)=f'(x_i)f'(x_j)\prod_{k\neq i,j}f(x_k)&\text{for $i,j=1,\ldots,N$ and $i\neq j$}.
\end{array}
$$
In particular on the diagonal ${\mathcal D}=\left\{x\in Q_{2R}\,:\,x_1=\ldots=x_N\right\}$ we deduce that
$$
D^2u(x)=f^{N-2}(x_1)\cdot M\left(f''(x_1)f(x_1),(f'(x_1))^2\right),
$$
where the matrix $M$ is given by $\eqref{matrix M}$.
Using Lemma \ref{linear algebra} with $a=f''(x_1)f(x_1)$ and $b=(f'(x_1))^2\geq 0$, we have 
$$
\Ppo(D^2u(x))=f^{N-2}(x_1)\cdot\big{(}f''(x_1)f(x_1)+(N-1)(f'(x_1))^2\big{)}\qquad\text{for $x\in{\mathcal D}$}.
$$
In particular 
\begin{equation}
\Ppo(D^2u(x))+\mu u(x)=0 \quad \text{for $x\in\mathcal D$}\\
\end{equation}
if and only if,
\begin{equation}\label{eq1}
\left\{
\begin{array}{rl}
f''(t)f(t)+(N-1)(f'(t))^2+\mu f^2(t)=0 & t\in(-R,R)\\
f(-R)=f(R)=0.& 
\end{array}
\right.
\end{equation}
Which is equivalent to
$$
\left\{
\begin{array}{rl}
(f^N)''(t)+N\mu f^N(t)=0 & t\in(-R,R)\\
f(-R)=f(R)=0& 
\end{array}
\right.
$$
and
\begin{equation}\label{eq2}
\mu=\frac1N\left(\frac{\pi}{2R}\right)^2,\qquad f(t)=\sqrt[N]{\cos(\frac{\pi}{2R}t)}.
\end{equation}
Now we need to prove that for such $f$, the function $u$ given by \eqref{multiplicative form} is in turn a solution in the 
whole cube $Q_{2R}$. By means of the representation formula \eqref{representation formulas} this is equivalent to show that
\begin{equation}\label{cvd}
\max_{v\in\Rn\backslash\left\{0\right\}}\frac{\langle D^2u(x)v,v\rangle}{|v|^2}=-\mu u(x)\qquad\forall x\in Q_{2R}.
\end{equation}
Let $x\in Q_{2R}$ and let $v=(f(x_1)\xi_1,\ldots,f(x_N)\xi_N)^\mathsf{T}$, with $\xi_1,\ldots,\xi_N\in\R$ and such that $|v|\neq0$. Then 
$$
\langle D^2u(x)v,v\rangle=\left[\sum_{i=1}^Nf''(x_i)f(x_i)\xi_i^2+2\sum_{i>j}f'(x_i)f'(x_j)\xi_i\xi_j\right]u(x)
$$
and using \eqref{eq1}
\begin{equation}\label{inequality}
\begin{split}
\langle D^2u(x)v,v\rangle&=-\left[\mu\sum_{i=1}^Nf^2(x_i)\xi_i^2+(N-1)\sum_{i=1}^N(f'(x_i))^2\xi_i^2-2\sum_{i>j}f'(x_i)f'(x_j)\xi_i\xi_j\right]u(x)\\
&=-\left[\mu|v|^2+\sum_{i>j}\left(f'(x_i)\xi_i-f'(x_j)\xi_j\right)^2\right]u(x)\\
&\leq-\mu u(x)|v|^2.
\end{split}
\end{equation}
Taking the supremum over $(\xi_1,\ldots,\xi_N)\neq (0,\ldots,0)$ we deduce that
$$
\max_{v\in\Rn\backslash\left\{0\right\}}\frac{\langle D^2u(x)v,v\rangle}{|v|^2}\leq-\mu u(x).
$$
Let
\begin{equation}\label{cond}
\tilde D=\{x\in Q_{2R}, \ \exists\ i_0\in\left\{1,\ldots,N\right\}\quad \text{s.t.}\quad \prod_{j\neq i_0}x_j\neq0\}.
\end{equation}Setting
$
\xi_i=\prod_{j\neq i}f'(x_j)
$,
let 
$$\hat v=(f(x_1)\prod_{j\neq 1}f'(x_j),\ldots,f(x_N)\prod_{j\neq N}f'(x_j))^\mathsf{T}.$$ 
 For $x\in \tilde D$, we have $\hat v\cdot e_{i_0}\neq0$ since the only zero of $f'(t)$ in $(-R,R)$ is $t=0$. In this way $|\hat v|>0$ and 
\begin{equation}\label{eq3}
\langle D^2u(x)\hat v,\hat v\rangle=-\mu u(x)|\hat v|^2.
\end{equation}
In view of \eqref{inequality}, \eqref{eq3}, then for every $x\in \tilde D$
\begin{equation}\label{eq3'}
\Ppo(D^2u(x))+\mu u(x)=0\,.
\end{equation} 
By continuity, see \eqref{continuity}, equality \eqref {eq3'} continues to be true in the whole cube. 
Summing up we have obtained the following 

\begin{theorem}\label{theorem1}
The  principal eigenvalue of $\Ppo$ in the cube $Q_{2R}$ is 
\begin{equation}\label{eigenvalue cube}
\mu^+_1=\frac{1}{N}\left(\frac{\pi}{2R}\right)^2
\end{equation}
and a corresponding principal eigenfunction if given by 
$\displaystyle u(x)=\prod_{i=1}^N\sqrt[N]{\cos\left(\frac{\pi}{2R}x_i\right)}$.
\end{theorem}

Conversely to what one could expect, the only cases in which  the eigenvalue problem 
\begin{equation}\label{eigenvalue problem}
\left\{\begin{array}{rl}
\Ppk(D^2u)+\mu u=0 & \text{in $Q_{2R}$}\\
u>0 & \text{in $Q_{2R}$}\\
u=0 & \text{on $\partial Q_{2R}$}
\end{array}\right.
\end{equation}
has a solution of type \eqref{multiplicative form} are $k=1$ and $k=N$. This is proved in the following 

\begin{theorem}
Let $2\leq k\leq N-1$ and let us assume that $u$ is a solution of \eqref{eigenvalue problem}. Then there are no functions $f\in C^2(-R,R)$ such that $u(x)=\prod_{i=1}^Nf(x_i)$. 
\end{theorem}
\begin{proof}
By contradiction let us assume that  $u(x)=\prod_{i=1}^Nf(x_i)$ is a solution of \eqref{eigenvalue problem}. 
Arguing as in the proof of Theorem \ref{theorem1}, we discover that $f$ must satisfy 
\begin{equation}\label{eq4}
kf''(t)f(t)+(N-k)(f'(t))^2+\mu f^{2}(t)=0\qquad\text{for $t\in(-R,R)$}.
\end{equation}
Hence $f(t)=(\cos(\frac{\pi}{2R}t))^{\frac kN}$ and $\mu=\frac{1}{N}\left(\frac{k\pi}{2R}\right)^2$. We claim that the function $$u(x)=\prod_{i=1}^N(\cos(\frac{\pi}{2R}x_i))^{\frac kN}$$ fails to be a solution of 
$$
\Ppk(D^2u(x))+\mu u(x)=0
$$ 
for some $x\in Q_{2R}\backslash\mathcal {D}$. Let $v_1,\ldots,v_k\in \RN$ be such that $v_i\cdot v_j=\delta_{ij}$ for $i,j=1,\ldots,k$.

Using the equation \eqref{eq4}, for $x\in Q_{2R}$:
\begin{equation}\label{eq5}
\begin{split}
\sum_{i=1}^k\left\langle D^2u(x)v_i,v_i\right\rangle&=\sum_{i=1}^k{\Big[}-\frac\mu k-\frac{N-k}{k}\sum_{l=1}^N\frac{(f'(x_l))^2}{f^2(x_l)}(v_i)_l^2\\
&\qquad\quad +2\sum_{l>m}\frac{f'(x_l)}{f(x_l)}\frac{f'(x_m)}{f(x_m)}(v_i)_l(v_i)_m{\Big]}u(x)\\
&=-\left[\mu+\sum_{i=1}^k\left\langle M\left(\frac{N-k}{k},-1\right)w_i,w_i\right\rangle\right]u
\end{split}
\end{equation}
where $$
w_i=\left(\frac{f'(x_1)}{f(x_1)}(v_i)_1,\ldots,\frac{f'(x_N)}{f(x_N)}(v_i)_N\right)^\mathsf{T}
$$ and $M(\frac{N-k}{k},-1)$, see Lemma \ref{linear algebra}, has eigenvalue $\frac{1-k}{k}N<0$ which is simple and $\frac Nk$ with multiplicity $N-1$. Now the idea is to choose $x\in Q_{2R}$ and $v_1,\ldots,v_k$ such that   
$w_1$ is in the eigenspace relative to $-\frac{k-1}{k}N$ and $w_2,\ldots,w_k$ are in the orthogonal eigenspace.\\
Let $\beta>\alpha>0$ be real fixed number . Let $x\in Q_{2R}$ such that $x_1=\ldots=x_{N-1}$ and  
$$
\frac{f'(x_1)}{f(x_1)}=\alpha\,\qquad\frac{f'(x_N)}{f(x_N)}=\beta.
$$
Note that such choice is possible since $\frac{f'(t)}{f(t)}=-\frac{k\pi}{2RN}\tan(\frac{\pi}{2R}t)$ maps the interval $(-R,R)$ onto $\R$.
Setting $\gamma^2=\frac{(\alpha\beta)^2}{(N-1)\beta^2+\alpha^2}$ we define
\begin{equation}\label{v1}
v_1=\gamma\left(\frac1\alpha,\ldots,\frac1\alpha,\frac1\beta\right)^\mathsf{T},
\end{equation}
so that $w_1=\gamma(1,\ldots,1)^\mathsf{T}$ and 
\begin{equation}\label{w1}
\left\langle Mw_1,w_1\right\rangle=-\frac{k-1}{k}N^2\gamma^2.
\end{equation}
Now we consider $k-1$ orthonormal vectors $v_2,\ldots,v_{k}$ of the $(N-2)$-dimensional subspace of $\RN$
$$
V=\left\{v\in\RN\,:\,(v)_1+\ldots+(v)_{N-1}=0,\;(v)_N=0\right\}.
$$ 
In this way $w_2,\ldots,w_k$ belong to the eigenspace relative to $\frac Nk$ 
and 
\begin{equation}\label{wk}
\sum_{i=2}^k\left\langle Mw_i,w_i\right\rangle=\frac{N(k-1)}{k}\alpha^2.
\end{equation}
Since by construction $\left\langle v_i,v_j\right\rangle=\delta_{ij}$ for any $i,j=1,\ldots,k$, we can use \eqref{w1}-\eqref{wk} in \eqref{eq5} to discover that
\begin{equation*}
\begin{split}
\Ppk(D^2u(x))&=\max_{\left\langle v_i,v_j\right\rangle=\delta_{ij}}\sum_{i=1}^k\left\langle D^2u(x)v_i,v_i\right\rangle\\
&=-\left[\mu+\min_{\left\langle v_i,v_j\right\rangle=\delta_{ij}}\sum_{i=1}^k\left\langle Mw_i,w_i\right\rangle\right]u(x)\\
&\geq-\left[\mu+\frac{N(k-1)}{k}\left(\alpha^2-N\gamma^2\right)\right]u(x)
\end{split}
\end{equation*}
and $\left(\alpha^2-N\gamma^2\right)$ is strictly negative by the choice $\beta>\alpha>0$. This contradicts the fact that $u$ is a solution of \eqref{eigenvalue problem} in the whole square.
\end{proof}

\subsection{Reversed baby Faber-Krahn inequality.} Let $R>0$ and let $\alpha=(\alpha_1,\ldots,\alpha_N)$ be such that $\alpha_i>0$ for any $i=1,\ldots,N$. We consider the  $N$-dimensional open rectangle with center $0$ and side lengths $2\alpha_i^{-1}R$, i.e.
$$
{\rm Rect}(\alpha)=\prod_{i=1}^N(-\alpha_i^{-1}R,\alpha_i^{-1}R). 
$$ 
Note that $$|{\rm Rect}(\alpha)|=(2R)^N\prod_{i=1}^N\frac{1}{\alpha_i}=|Q_{2R}|$$
if, and only if, $\prod_{i=1}^N\alpha_1=1$. We are going to show that
\begin{equation}\label{FK}
\textbf{\lq\lq\ The  cube has the largest $\mu_1^+$ among  rectangles with a given measure\ \rq\rq}.\tag{FK}
\end{equation}
\begin{theorem}\label{Faber-Krahn}
Let $\alpha=(\alpha_1,\ldots,\alpha_N)$ be such that $\alpha_i>0$ for $i=1,\ldots,N$. Then the principal eigenvalue of $\Ppo$ in ${\rm Rect}(\alpha)$ is 
\begin{equation}\label{eigenvalue rectangle}
\mu^+_1=\frac{1}{\displaystyle{\frac{1}{\alpha_1^2}+\ldots+\frac{1}{\alpha_N^2}}}\left(\frac{\pi}{2R}\right)^2.
\end{equation}
Moreover there exists $p=(p_1,\ldots,p_N)$, $p_i>-1$ for any $i=1,\ldots,N$, such that 
$$
u(x)=\prod_{i=1}^N\left(\cos(\frac{\pi}{2R}\alpha_ix_i)\right)^{\frac{1}{p_i+1}}
$$
is a principal eigenfunction.
\end{theorem}

Before giving the proof of the theorem let us explicitly remark that, in view of \eqref{eigenvalue cube}, \eqref{eigenvalue rectangle}, the statement \eqref{FK}
reduce to the well know inequality between  harmonic mean and geometric mean:
$$
\frac{N}{\displaystyle{\frac{1}{\alpha_1^2}+\ldots+\frac{1}{\alpha_N^2}}}\leq\sqrt[N]{\alpha_1^2\cdot\ldots\cdot\alpha_N^2}\;.
$$  
The equality occurs if, and only if, the rectangle is a cube. Moreover it is worth to point out that from \eqref{eigenvalue rectangle} we immediately deduce that the infimum of $\mu^+_1$ among all domains with fixed measure is zero.

\begin{proof}[Proof of Theorem \ref{Faber-Krahn}] For $p=(p_1,\ldots,p_N)$ to be fixed and $i=1,\ldots,N$, let us consider the functions
$$
f_i(t)={\left(\cos(\frac{\pi}{2R}t)\right)}^{\frac{1}{p_i+1}}\qquad t\in(-R,R).
$$
Note that 
$$
(f_i^{p_i+1})''(t)+\left(\frac{\pi}{2R}\right)^2f^{p_i+1}(t)=0 \qquad t\in(-R,R)
$$
which yield
\begin{equation}\label{eq6}
(p_i+1)f_i(t)f_i''(t)+p_i(p_i+1)(f_i'(t))^2+\left(\frac{\pi}{2R}\right)^2f_i^2(t)=0.
\end{equation}
Set
$$
u(x)=\prod_{i=1}^Nf_i(\alpha_i x_i)\qquad x\in{\rm Rect}(\alpha),
$$
hence 
$$
\begin{array}{ll}
\displaystyle\partial_{ii}u(x)=\alpha_i^2f_i''(\alpha_ix_i)\prod_{k\neq i}f_k(\alpha_kx_k)&\text{for $i=1,\ldots,N$}\\
\displaystyle\partial_{ij}u(x)=\alpha_i\alpha_jf_i'(\alpha_ix_i)f_j'(\alpha_jx_j)\prod_{k\neq i,j}f_k(\alpha_kx_k)&\text{for $i,j=1,\ldots,N$ and $i\neq j$}.
\end{array}
$$
For $x\in{\rm Rect}(\alpha)$ and $v=(v_1,\ldots,v_N)^\mathsf{T}$ such that $|v|\neq0$ one has
$$
\left\langle D^2u(x)v,v\right\rangle=\sum_{i=1}^N\alpha_i^2f_i''(\alpha_ix_i)\prod_{k\neq i}f_k(\alpha_kx_k)v_i^2+2\sum_{i>j}\alpha_i\alpha_jf_i'(\alpha_ix_i)f_j'(\alpha_jx_j)\prod_{k\neq i,j}f_k(\alpha_kx_k)v_iv_j.
$$
Setting $v_i=f_i(\alpha_ix_i)\xi_i$, $\xi_i\in\R$ for $i=1,\ldots,N$, the previous equality reads as
$$
\left\langle D^2u(x)v,v\right\rangle=\left[\sum_{i=1}^N\alpha_i^2f_i''(\alpha_ix_i)f_i(\alpha_ix_i)\xi_i^2+2\sum_{i>j}\alpha_i\alpha_jf_i'(\alpha_ix_i)f_j'(\alpha_jx_j)\xi_i\xi_j\right]u(x).
$$
Now, using \eqref{eq6}, we obtain
\begin{equation}\label{eq7}
\begin{split}
\left\langle D^2u(x)v,v\right\rangle&=-{\Bigg[}\left(\frac{\pi}{2R}\right)^2\sum_{i=1}^N\frac{\alpha_i^2}{p_i+1}v_i^2+\sum_{i=1}^Np_i(f'_i(\alpha_ix_i)\alpha_i\xi_i)^2\\&\hspace{3cm}-2\sum_{i>j}\alpha_i\alpha_jf_i'(\alpha_ix_i)f_j'(\alpha_jx_j)\xi_i\xi_j{\Bigg]}u(x)\\
&=-\left[\left(\frac{\pi}{2R}\right)^2\sum_{i=1}^N\frac{\alpha_i^2}{p_i+1}v_i^2+\left\langle Mw,w\right\rangle\right]u(x),
\end{split}
\end{equation}
where 
\begin{equation}\label{w}
w=(f'_1(\alpha_1x_1)\alpha_1\xi_1,\ldots,f'_N(\alpha_Nx_N)\alpha_N\xi_N)^\mathsf{T}
\end{equation}
and 
$$M=
\left(
\begin{array}{cccc}
p_1 & -1 & \ldots & -1\\
-1 & p_2 & \ldots & -1\\
\vdots &  \vdots & \ddots & \vdots\\
-1 & \ldots & -1 & p_N
\end{array}
\right).
$$
Our aim is now  to prove that there exist $p_1,\ldots,p_N$ and a positive constant $\kappa$ such that 
\begin{equation}\label{claim}
\begin{split}
\frac{\alpha_i^2}{p_i+1}&=\frac1\kappa\qquad \text{for $i=1,\ldots,N$}\\
\lambda_1(M)=0&\leq\lambda_2(M)\leq\ldots\leq\lambda_N(M).
\end{split}
\end{equation}
Since $p_i=\kappa\alpha_i^2-1$ we obtain
$$
M=-\left(
\begin{array}{ccc}
1 &\ldots & 1\\
\vdots & \cdots & \vdots\\
1 &\ldots & 1
\end{array}
\right)
+\kappa\,\text{diag}(\alpha_1^2,\ldots,\alpha_N^2).
$$
Hence for any $w=(w_1,\ldots,w_N)^\mathsf{T}$
$$
\left\langle Mw,w\right\rangle=-(w_1+\ldots+w_N)^2+\kappa (\alpha_1^2w_1^2+\ldots+\alpha_N^2w_N^2)
$$
and \eqref{claim} follows by taking 
\begin{equation}\label{max}
\kappa=\max_{|w|\neq0}\frac{(w_1+\ldots+w_N)^2}{\alpha_1^2w_1^2+\ldots+\alpha_N^2w_N^2}
={\displaystyle{\frac{1}{\alpha_1^2}+\ldots+\frac{1}{\alpha_N^2}}}.
\end{equation}
Coming back now to \eqref{eq7}, we deduce that
\begin{equation*}
\begin{split}
\left\langle D^2u(x)v,v\right\rangle&=-\left[\left(\frac{\pi}{2R}\right)^2\frac1\kappa |v|^2+\left\langle Mw,w\right\rangle\right]u(x)\\
&\leq-\left(\frac{\pi}{2R}\right)^2\frac1\kappa|v|^2 u(x) \qquad\text{for any $v\in\RN$}.
\end{split}
\end{equation*}
Moreover the equality $\left\langle D^2u(x)v,v\right\rangle=-\left(\frac{\pi}{2R}\right)^2\frac1\kappa|v|^2$ is achieved if $w$, which is given by \eqref{w}, realize the maximum in \eqref{max}. Then if $f'_i(\alpha_ix_i)\neq0$ it is sufficient to take $\xi_1,\ldots,\xi_N$ such that $f'_i(\alpha_ix_i)\alpha_i\xi_i=\frac{1}{\alpha_i^2}$. Since $f_i'(t)=0$ implies $t=0$, we deduce that 
\begin{equation}\label{eq9}
\Ppo(D^2u(x))=\max_{v\in\RN\backslash\left\{0\right\}}\frac{\left\langle D^2u(x)v,v\right\rangle}{|v|}^2=-\frac{1}{\displaystyle{\frac{1}{\alpha_1^2}+\ldots+\frac{1}{\alpha_N^2}}}\left(\frac{\pi}{2R}\right)^2u(x)\qquad \text{if \;\;$\prod_{i=1}^Nx_i\neq0$}.
\end{equation}
By continuity the equality \eqref{eq9} still holds in the whole ${\rm Rect}(\alpha)$.
\end{proof}
 
The previous results show that the  behavior of the principal eigenvalues $\mu(\Delta)$ of the Laplacian $\Delta$ and $\mu^+_1$ of $\Ppo$ is opposite with respect to the symmetry of the domain, at least for square type domains. Note that in ${\rm Rect}(\alpha)$ one has $\mu(\Delta)=(\frac{\pi}{2R})^2\sum_{i=1}^N\alpha_i^2$, while $\mu^+_1=(\frac{\pi}{2R})^2\left(\sum_{i=1}^N\frac{1}{\alpha_i^2}\right)^{-1}$.  This surprising feature can be further strengthened: 
\begin{equation}\label{FK2}
\textbf{\lq\lq\ The ball has a larger principal eigenvalue than the cube having the same measure\ \rq\rq}.\tag{FK2}
\end{equation}
Let us consider  the ball $B_\rho$ of radius $\rho>0$. We know, by \eqref{Pk-Delta}, that 
$$
\mu^+_1(B_\rho)=\left(\frac{\pi}{2\rho}\right)^2,
$$
with $u(x)=\cos(\frac{\pi}{2\rho}|x|)$ as principal eigenfunction. Now if we fix the measure, say equals to $(2R)^N$,
and we take $\rho=2R\omega_N^{-\frac{1}{N}}$, being $\omega_N$ the measure of the unit ball in $\RN$, then
$$
\left|B_\rho\right|=|Q_{2R}|
$$
and 
$$
\mu_1^+(B_\rho)=\left(\frac{\pi}{2R}\right)^2\frac{\omega_N^{\frac{2}{N}}}{4}>\frac{1}{N}\left(\frac{\pi}{2R}\right)^2=\mu_1^+(Q_{2R}).
$$

\subsection{On the principal eigenvalue for the intersection of rectangles} As was said in the introduction, in \cite{L} Lieb showed that if $A, B\subset\Rn$ are two bounded domains, then 
\begin{equation}\label{L1}
\inf_{x\in\Rn}\mu\left(\Delta,A\cap B_x\right)<\mu(\Delta,A)+\mu(\Delta, B).
\end{equation}
We now show that the inequality \eqref{L1} is not true in general for $\mu^+_1$ , actually it is reversed if $A$ and $B$ are some specific rectangles. Let us assume $N=2$ for simplicity and let
$$
A=\left(-\frac{R}{\alpha_1},\frac{R}{\alpha_1}\right)\times\left(-\frac{R}{\alpha_2},\frac{R}{\alpha_2}\right),\quad B=\left(-\frac{R}{\alpha_2},\frac{R}{\alpha_2}\right)\times\left(-\frac{R}{\alpha_1},\frac{R}{\alpha_1}\right).
$$
Without loss of generality we may suppose $\alpha_1\leq\alpha_2$. Then using \eqref{eigenvalue cube} one has
$$
\inf_{x\in \Rn}\mu^+_1(A\cap B_x)=\mu^+_1\left(\left(-\frac{R}{\alpha_2},\frac{R}{\alpha_2}\right)^2\right)=\frac{\alpha_2^2}{2}\left(\frac{\pi}{2R}\right)^2,
$$
whereas 
$$
\mu^+_1(A)=\mu^+_1(B)=\frac{\alpha_1^2\alpha_2^2}{\alpha_1^2+\alpha_2^2}\left(\frac{\pi}{2R}\right)^2
$$
in view of \eqref{eigenvalue rectangle}. In this way if we choose $\alpha_2^2>3\alpha_1^2$, then
\begin{equation}\label{L2}
\inf_{x\in\Rn}\mu^+_1(A\cap B_x)>\mu^+_1(A)+\mu^+_1(B).
\end{equation}
In higher dimension, $N\geq3$, let us consider $\alpha=(\alpha_1,\alpha_2,\ldots,\alpha_N)$, $\tilde\alpha=(\alpha_2,\alpha_1,\ldots,\alpha_N)$ and $\tilde{\tilde\alpha}=(\alpha_2,\alpha_2,\ldots,\alpha_N)$ with $0<\alpha_1\leq\ldots\leq\alpha_N$. Set
\begin{equation*}
A={\rm Rect}(\alpha),\quad B={\rm Rect}(\tilde\alpha).
\end{equation*}
Then 
$$
\inf_{x\in\Rn}\mu^+_1(A\cap B_x)=\mu^+_1\left({\rm Rect}(\tilde{\tilde\alpha})\right)=\frac{1}{\frac{2}{\alpha_2^2}+\sum_{i=3}^N\frac{1}{\alpha_i^2}}
$$
and so \eqref{L2} is satisfied by choosing $\alpha\in\Rn$ in such a way
$$
\frac{1}{\alpha_1^2}>\frac{3}{\alpha_2^2}+\sum_{i=3}^N\frac{1}{\alpha_i^2}.
$$

\section{Application: H\"older continuity in convex domains}  \label{regh}
We  study the global H\"older continuity of viscosity solutions of 
\begin{equation}\label{Dirichlet}
\left\{
\begin{array}{rl}
\Ppo(D^2u)=f(x) & \text{in $\Omega$}\\
u=0\;\;\;\;\; & \text{on $\partial \Omega$}
\end{array}
\right.
\end{equation}
where $\Omega\subset\Rn$ is a bounded convex domain and $f$ is a continuous and bounded function in $\Omega$.

For notational simplicity let $Q\equiv Q_\pi$ and $Q(y)$ be respectively the $N$-dimensional open cubes with centers $0$ and $y\in\Rn$  and side length $\pi$, i.e. 
$$
Q=\left(-\frac\pi2,\frac\pi2\right)^N,\qquad Q(y)=\prod_{i=1}^N\left(y_i-\frac\pi2,y_i+\frac\pi2\right).
$$
By convexity and rescaling, the domain  $\Omega$ may be expressed as \lq\lq intersection of cubes\rq\rq\ of side length $\pi$: there exist a subset ${\mathcal C}$ of $Y\times{\mathcal O}$, ${\mathcal O}$ being the set of $N\times N$ orthogonal matrices and $Y\subset\Rn$, such that 
\begin{equation}\label{representation}
\Omega=\bigcap_{(y,O)\in {\mathcal C}}O Q(y),
\end{equation}
where $OQ(y)=\left\{Ox\,:\, x\in Q(y)\right\}$.\\
Let us denote by $\phi(x)=\displaystyle\prod_{i=1}^N\sqrt[N]{\cos x_i}$ the eigenfunction provided Theorem \ref{theorem1}. Note that for any $(y,O)\in {\mathcal C}$, the function $\phi_{y,O}(x)=\phi(O^\mathsf{T}x-y)$ solves 
\begin{equation}\label{eigenfunctiony}
\left\{
\begin{array}{rl}
\Ppo(D^2\phi_{y,O})+\frac{1}{N}\phi_{y,O}=0 & \text{in $OQ(y)$}\\
\phi_{y,O}=0& \text{on $\partial(OQ(y))$}.
\end{array}
\right.
\end{equation} 
Moreover for any $x,z\in \overline{OQ(y)}$ one has
\begin{equation}\label{Holder condition}
\left|\phi_{y,O}(x)-\phi_{y,O}(z)\right|\leq\left(\sqrt{N}|x-z|\right)^{\frac{1}{N}}.
\end{equation}

\begin{theorem}[\textbf{H\"older}]\label{holder}
Let $\Omega$ be given by \eqref{representation}.
If there exist  $\alpha>0$ and $\beta\in(0,1]$ such that 
\begin{equation}\label{condition}
f(x)\geq-\alpha \left(\inf_{{\mathcal C}}\phi_{y,O}(x)\right)^\beta\qquad\forall x\in\Omega,
\end{equation}
then there exits a unique viscosity solution $u$ of \eqref{Dirichlet}. Moreover $u\in C^{0, \frac\beta N}(\overline\Omega)$ and  the H\"older norm of $u$ depends only on $\alpha$, $\beta$, $N$ and the $L^\infty$ norms of $u$ and $f$.  
\end{theorem}
\begin{proof}
The existence and uniqueness of $u$ follows from Perron's method,  \cite[Theorem 4.1]{CIL}. For this, note that the construction of a continuous subsolution $\underline u$ of \eqref{Dirichlet}, with general $f$, is standard under the uniform exterior sphere (or cone) condition of $\Omega$. On the other hand, owing to the degeneracy of the operator $\Ppo$ with respect inf type operations,   the equivalent argument used for subsolutions actually fails for the construction of  supersolutions null on the boundary.  This is the point where the assumption \eqref{condition} is used. 
For any $x\in\overline\Omega$ let
$$
\overline u(x)=\frac{N\alpha}{\beta}\inf_{\mathcal C}\phi^\beta_{y,O}(x).
$$
Using \eqref{Holder condition}, then for any $x,z\in\overline\Omega$ one has
\begin{equation*}
\begin{split}
\left|\overline u(x)-\overline u(z)\right|&\leq \frac{N\alpha}{\beta}\sup_{\mathcal C}\left|\phi^\beta_{y,O}(x)-\phi^\beta_{y,O}(z)\right|\\
&\leq \frac{N\alpha}{\beta}\sup_{\mathcal C}\left|\phi_{y,O}(x)-\phi_{y,O}(z)\right|^\beta\leq \frac{N\alpha}{\beta}\left(\sqrt{N}|x-z|\right)^{\frac{\beta}{N}}.
\end{split}
\end{equation*}
Hence $\overline u\in C^{0,\frac\beta N}(\overline\Omega)$. Moreover for any $(y,O)\in\mathcal C$ and any $x\in\Omega$
\begin{equation}\label{4eq1}
\begin{split}
\Ppo\left(D^2\frac{N\alpha}{\beta}\phi^\beta_{y,O}(x)	\right)&\leq N\alpha(\beta-1)\phi_{{y,O}}^{\beta-2}(x)\Pmo(D\phi_{y,O}(x)\otimes D\phi_{y,O}(x))\\&\;\quad+N\alpha \phi_{y,O}^{\beta-1}(x)\Ppo(D^2\phi_{{y,O}}(x))\\
&= N\alpha\phi_{y,O}^{\beta-1}(x)\Ppo(D^2\phi_{{y,O}}(x))\\
&=-\alpha\phi_{y,O}^\beta(x)\\& \leq f(x).
\end{split}
\end{equation}
Then $\overline u$, which is the infimum of all  $\phi_{y,O}$, is in turn a supersolution of \eqref{Dirichlet}. Moreover  $\overline u=0$ on $\partial\Omega$. Hence the Perron's method provides existence and uniqueness for \eqref{Dirichlet}.

Let us prove now that the solution $u\in C^{0, \frac\beta N}(\overline\Omega)$. Without loss of generality we may assume $u\not\equiv0$. \\
Let $\Delta_\delta=\left\{(x,y)\in\Omega\times\Omega\,:\,|x-y|<\delta\right\}$ where $\delta$ is a positive number  such that 
\begin{equation}\label{delta}
2\left\|u\right\|_\infty\frac\beta N\left(1-\frac\beta N\right)\delta^{-2}>\left\|f\right\|_\infty.
\end{equation}
Set $M=\max \left(\frac{N^{1+\frac{\beta}{2N}}\alpha}{\beta},\frac{2\left\|u\right\|_\infty}{\delta^\frac\beta N}\right)$.  We assume by contradiction that 
\begin{equation}\label{eq10}
0<\max_{\overline\Delta_\delta}\left\{u(x)-u(y)-M|x-y|^\frac\beta N\right\}=u(x_0)-u(y_0)-M|x_0-y_0|^\frac\beta N.
\end{equation}
In particular $x_0\neq y_0$. If $|x_0-y_0|=\delta$ then
\begin{equation*}
0<u(x_0)-u(y_0)-M|x_0-y_0|^\frac\beta N\leq2\left\|u\right\|_\infty- M\delta^\frac\beta N\leq0,
\end{equation*}
by the choice of $M$. 
If $y_0\in\partial\Omega$, then there exist $(\tilde{y}_0,\tilde O)\in{\mathcal C}$ such that $\Omega\subset \tilde OQ(\tilde{y}_0)$ and $y_0\in\partial(\tilde OQ(\tilde{y}_0))$. As in \eqref{4eq1},  the function $\psi(x)=\frac{N\alpha}{\beta}\phi_{\tilde{y}_0,\tilde O}^\beta(x)$ satisfies in $\Omega$ the inequality  
$\Ppo(D^2\psi)\leq f(x)$.
Moreover $\psi\geq0$ on $\partial\Omega$. By comparison $u\leq\psi$ in $\overline\Omega$, hence 
\begin{equation}\label{contr}
0<u(x_0)-u(y_0)-M|x_0-y_0|^{\frac\beta N}\leq \psi(x_0)-\psi(y_0)-M|x_0-y_0|^{\frac\beta N}\leq0,
\end{equation}
in view of \eqref{Holder condition} and  the choice of $M$.
The above contradictions imply that  $(x_0,y_0)\in\Delta_\delta$ or $(x_0,y_0)\in\partial\Omega\times\Omega$. From \eqref{eq10} we deduce that $u-\varphi$ has a local minimum at $y_0$ with $\varphi(y)=-M|y-x_0|^\frac\beta N$. Then  
\begin{equation*}
f(y_0)\geq\Ppo(D^2\varphi(y_0))\geq M\frac\beta N\left(1-\frac\beta N\right)\delta^{\frac\beta N-2}\geq 2\left\|u\right\|_\infty\frac\beta N\left(1-\frac\beta N\right)\delta^{-2}
\end{equation*}
and this contradicts \eqref{delta}.  
\end{proof}

\begin{remark}
Following the argument of the previous proof and looking at \eqref{contr},  it is clear that the global $C^{0,\gamma}$ H\"older continuity, $\gamma\in(0,1)$, is still true for any  nonpositive  supersolutions $u$ of \eqref{Dirichlet} without assuming the convexity of $\Omega$
 (note that $u\leq0$ forces $f$ to be nonnegative somewhere). On the other hand the global regularity fails if instead we consider nonnegative supersolutions.  For example let us consider the nonnegative continuous function
$$
u(x)=\left\{
\begin{array}{cl}
\displaystyle\frac{1}{\sigma-\sum_{i=1}^N\log(\cos x_i)} & \text{if $x\in Q$}\\
0 & \text{if $x\in\partial Q$}.
\end{array}
\right.
$$
We are going to choose $\sigma$ in such a way that $u$ be concave, in particular $\Ppo(D^2u)\leq0$ in $Q$. 
Let $v(x):=\sum_{i=1}^{N}g(x_i):=\sum_{i=1}^N\log(\cos x_i)$, then $\nabla u=\frac{1}{(\sigma-v)^2}\nabla v$ while
$$D^2u=\frac{1}{(\sigma-v)^3}(2\nabla v\otimes\nabla v+(\sigma-v)D^2v).$$ And then, since $(D^2v)_{ij}=\delta_{ij}g''(x_i)$, for any $w\in\RN$
\begin{equation*}
\begin{split}
\left\langle D^2u(x)w,w\right\rangle& =\frac{1}{(\sigma-v)^3}(2(\nabla v\cdot w)^2+(\sigma-v)\sum_{i=1}^Ng''(x_i)w_i^2)\\
&=\frac{1}{(\sigma-v)^3}\left(2\left(\sum_{i=1}^N\tan(x_i)w_i\right)^2-(\sigma-v)\sum_{i=1}^N\frac{1}{\cos^2(x_i)}w_i^2\right)\\
&\leq\frac{1}{(\sigma-v)^3}((2N-\sigma)\sum_{i=1}^N\frac{1}{\cos^2(x_i)}w_i^2)\\
&=0\qquad\text{if $\sigma=2N$}.
\end{split}
\end{equation*}
 On the other hand for any $\gamma\in(0,1]$
$$\sup_{\substack{x,y\in \overline Q\\x\neq y}}\frac{|u(x)-u(y)|}{|x-y|^\gamma}=+\infty.$$
\end{remark}

\end{document}